\documentclass[reqno]{amsart}

\usepackage{amsfonts,amssymb, amscd, latexsym, graphicx, psfrag, color,float}
\usepackage[normalem]{ulem}

\usepackage{dbnsymb}

\usepackage[all, knot, poly]{xy}

\xyoption{knot}

\usepackage{a4wide}

\usepackage{hyperref}
\usepackage[mathscr]{euscript}
\usepackage{mathtools}
\usepackage{subfigure}
\usepackage{calc}
\usepackage{tikz}

\usepackage{ifthen}

\usetikzlibrary{cd,arrows,matrix}
\usepackage{newtxtext}
\usepackage{newtxmath}
\usepackage{anyfontsize}
\usepackage{bm}

\usepackage[english]{babel}

\usepackage[makeroom]{cancel}

\hypersetup{colorlinks=true,linkcolor=blue}

\usepackage{enumitem}

\usepackage{calligra}

\usepackage{mathrsfs}

\usepackage{bbm}

\numberwithin{equation}{subsection}
\numberwithin{figure}{subsection}

\newtheorem{dummy}{dummy}[section]
\newtheorem{lemma}[dummy]{Lemma}
\newtheorem{theorem}[dummy]{Theorem}

\newtheorem{conjecture}[dummy]{Conjecture}
\newtheorem{corollary}[dummy]{Corollary}
\newtheorem{proposition}[dummy]{Proposition}

\newtheorem{definition/proposition}[dummy]{Definition/Proposition}

\theoremstyle{definition}
\newtheorem{definition}[dummy]{Definition}

\newtheorem{remark}[dummy]{Remark}

\DeclareMathOperator{\cone}{Cone}

\DeclareMathOperator{\coker}{Coker}

\DeclareMathOperator{\Gr}{Gr}

\newcommand{\bfL}{\mathbf{L}}

\newcommand{\cA}{\mathcal{A}}

\newcommand{\id}{\mathsf{id}}

\renewcommand{\emptyset}{\varnothing}

\renewcommand{\tilde}{\widetilde}

\newcommand{\field}{\mathbb{K}}

\newcommand{\isomorphic}{\cong}

\newcommand{\CH}{\mathrm{CH}}

\newcommand{\modulispace}{\mathcal{M}}

\newcommand{\Hom}{\mathrm{Hom}}

\makeatletter
\newcommand{\setword}[2]{%
  \phantomsection
  #1\def\@currentlabel{\unexpanded{#1}}\label{#2}%
}
\makeatother

\newcommand{\SmProj}{\mathbf{SmProj}}

\newcommand{\Sm}{\mathbf{Sm}}

\newcommand{\Var}{\mathbf{Var}}

\newcommand{\prop}{\text{\tiny$\mathrm{prop}$}}

\newcommand{\Cor}{\mathbf{Cor}}

\newcommand{\rat}{\mathrm{rat}}

\newcommand{\FinCor}{\mathbf{FinCor}}

\newcommand{\PreSh}{\mathbf{PreSh}}

\newcommand{\Ab}{\mathcal{A}b}

\newcommand{\Sh}{\mathbf{Sh}}

\newcommand{\tr}{\mathrm{tr}}

\newcommand{\Nis}{\mathrm{Nis}}

\newcommand{\eff}{\mathrm{eff}}

\newcommand{\gm}{\mathrm{gm}}

\newcommand{\DM}{\mathbf{DM}}

\newcommand{\Chow}{\mathbf{Chow}}

\newcommand{\M}{\mathrm{M}}

\newcommand{\pt}{\mathrm{pt}}

\newcommand{\weightcomplex}{\mathrm{W}_{\mathrm{c}}}

\newcommand{\weight}{\mathrm{W}}

\newcommand{\realization}{\mathfrak{R}}

\newcommand{\Ch}{\mathbf{Ch}}

\newcommand{\Coch}{\mathbf{Coch}}

\newcommand{\Ho}{\mathrm{Ho}}

\newcommand{\op}{\mathrm{op}}

\begin{document}

\title[Integral cohomology of dual boundary complexes is motivic]{Integral cohomology of dual boundary complexes is motivic}

\author[T. Su]{Tao Su}
\email{sutao@bimsa.cn}
\address{Beijing Institute of Mathematical Sciences and Applications}

\date{}

\begin{abstract}

In this note, we give a motivic characterization of the integral cohomology of dual boundary complexes of smooth quasi-projective complex algebraic varieties. 
As a corollary, the dual boundary complex of any variety stably isomorphic to $\mathbb{G}_m^r\times \mathbb{A}^s$ with $s>0$ is contractible.
In a separate paper \cite{Su26}, this corollary has been used by the author in his proof of the weak geometric P=W conjecture for very generic $GL_n(\mathbb{C})$-character varieties over any punctured Riemann surfaces.

\end{abstract}

\maketitle

\tableofcontents

\addtocontents{toc}{\protect\setcounter{tocdepth}{1}}

\section*{Introduction}

Let $X$ be a smooth complex quasi-projective variety. 
Take a log compactification $\overline{X}$ with simple normal crossing divisor $\partial X=\overline{X}\setminus X$. 
The dual boundary complex $\mathbb{D}\partial X$ is the dual complex of the irreducible components of $\partial X$, encoding how the irreducible components intersect.
If necessary, we refer to \cite{Su26} for a crash review.
Up to homotopy, $\mathbb{D}\partial X$ is an invariant of $X$ as an algebraic variety \cite{Dan75}, independent of the choice of the log compactifications.
In general, the homotopy type of any finite simplicial complex can be realized by a dual boundary complex \cite[E.g.2.6]{Pay13}. 
On the contrary, a conjecture attributed to M. Kontsevich (and indepdently, J. Koll\'{a}r and C. Xu) states the following:
\begin{conjecture}[Kontsevich(-Koll\'{a}r-Xu)]\label{conj:Kontsevich_conjecture}
The dual boundary complex of any log \textbf{CY} variety is (a finite quotient of) a sphere.
\end{conjecture}
The conjecture is termed as the Algebra-geometric version of the Poincar\'{e} conjecture in \cite[Conj.1.2]{MM24}.
So far, the main general result is due to Koll\'{a}r-Xu \cite{KX16}, which proves the finite quotient version of the conjecture in dimension at most $5$.
In higher dimensions, only partial results are known \cite{Mau20,Bra21,Mor21,FMM22,MM24}, see \cite[\S1.6]{MM24} for a quick summary.
Among them, we recall that the rational cohomology of the dual boundary complex is well understood due to an intrinsic characterization via the mixed Hodge structure on the bulk variety (see e.g. \cite{Pay13}):
\[
\tilde{H}^{i-1}(\mathbb{D}\partial X;\mathbb{Q}) \isomorphic \Gr_0^WH_c^i(X;\mathbb{Q}).
\]
On the other hand, the torsion of the integral cohomology remains a mystery (see \cite[P.529]{KX16}, \cite[\S1.6.(2)]{MM24}).
Our main purpose in this note is to enhance the previous characterization to a motivic level so as to encode integral cohomology:
\begin{proposition}\label{prop:dual_boundary_complex_is_motivic}
For any smooth complex quasi-projective variety $X$, we have
\begin{equation}
\tilde{H}^{i-1}(\mathbb{D}\partial X;\mathbb{Z}) \isomorphic H_{\weight,c}^{i,0}(X;\mathbb{Z}),
\end{equation}
where $H_{\weight,c}^{a,b}(X;\mathbb{Z})$ is the \emph{integral singular weight cohomology with compact support} of $X$ (Definition \ref{def:Betti_weight_cohomology_with_compact_support}).
\end{proposition}

An important class of examples for Conjecture \ref{conj:Kontsevich_conjecture} comes from the geometric P=W conjecture \cite{KNPS15,Sim16}. 
By a folklore conjecture, the smooth character variety $\modulispace_B$ over any punctured Riemann surfaces is expected to be log Calabi-Yau. 
See \cite{Wha20,FF23} for the proof for $SL_2(\mathbb{C})$-character varieties.
The weak version of the geometric P=W conjecture then becomes a special case of the conjecture above: the dual boundary complex of $\modulispace_B$ has the homotopy type of a sphere
of one dimension lower.
For previous results for lower genera or lower ranks, see \cite{Kom15,Sim16,Sza19,Sza21,Su25,NS22,MMS22,FF23}.

Our second goal is to pave the way for a separate paper \cite{Su26}, in which we prove the weak geometric P=W conjecture for all very generic character varieties.
That is, we prove the following corollary:
\begin{corollary}\label{cor:dual_boundary_complex_of_a_stably_torus_times_affine_space_is_contractible}
 The dual boundary complex of any variety stably isomorphic to $\mathbb{G}_m^r\times \mathbb{A}^s$ with $s>0$ is contractible.
\end{corollary}
Here, a \emph{stable isomorphism} between two varieties means
an isomorphism after taking products of both varieties with
the \emph{same} affine space $\mathbb{A}^{\ell}$, for some
$\ell\geq 0$.
A \emph{stably affine space} is a variety stably isomorphic
to an affine space $\mathbb{A}^d$.

The question whether every complex stably affine space $Y$
of dimension $d\geq 3$ is isomorphic to $\mathbb{C}^d$
is a stable form of the \emph{Zariski cancellation problem}.
This question was open when the first version of this paper
appeared in~2024.
In positive characteristic, counterexamples exist in every
dimension $d\geq 3$ \cite{Gup14a,Gup14b}.

\section{Preliminaries on motives}

\subsection{Chow and geometric motives}\label{subsec:Chow_and_geometric_motives}
We first make some preparations.

\subsubsection{Chow motives}\label{sec:Chow_motives}


We introduce some notations for Chow motives:
\begin{enumerate}[wide,labelwidth=!,labelindent=0pt]
\item
Let $\SmProj(\field)\subset\Sm(\field)\subset\Var(\field)$ be the category of smooth projective varieties, smooth varieties, and varieties over $\field$, respectively.
Let $\Cor_{\rat}(\field)$ be the category of (deg $0$) rational correspondences over $\field$, whose objects are the same as $\SmProj(\field)$, but with morphisms
\begin{equation}\label{eqn:rational_correspondences}
\Hom_{\Cor_{\rat}(\field)}(X,Y)=\oplus_{j\in J}\CH^{\dim Y_j}(Y_j \times X),
\end{equation}
where $Y_j, j\in J$ are the connected components of $Y$, and $\CH^{\bullet}$ stands for the Chow cohomology.

\item
Let $\Chow_{\rat}^{\eff}(\field)$ be the category of (covariant) effective pure Chow motives over $\field$, defined as the idempotent completion of $\Cor_{\rat}(\field)$.
Then, we have a canonical covariant graph functor
\[
\M_{\rat}:\SmProj(\field)\rightarrow \Chow_{\rat}^{\eff}(\field): X\mapsto \M_{\rat}(X)=(X,[\id_X]=[\Delta_X]),\quad (f:X\rightarrow Y)\mapsto [\Gamma_f],
\]

\item
The opposite category $\Chow_{\rat}^{\eff}(\field)^{\op}$ is termed as the category of \emph{contravariant effective pure Chow motives}.
Let $K^b(\Chow_{\rat}^{\eff}(\field)^{\op})=\Ho(\Coch^b(\Chow_{\rat}^{\eff}(\field)^{\op}))$ be the homotopy category of bounded cochain complexes in $\Chow_{\rat}^{\eff}(\field)^{\op}$.

\noindent{}To clarify our notations, we make the following \textbf{Convention} \setword{{\color{blue}$1$}}{convention:weight_complex}: 
\begin{itemize}[wide]
\item
By reversing the arrows, we can always identify $K^b(\Chow_{\rat}^{\eff}(\field)^{\op})$ with the homotopy category 
$K_b(\Chow_{\rat}^{\eff}(\field))=\Ho(\Ch_b(\Chow_{\rat}^{\eff}(\field)))$ of bounded chain complexes in $\Chow_{\rat}^{\eff}(\field)$:
\[
K_b(\Chow_{\rat}^{\eff}(\field))\isomorphic K^b(\Chow_{\rat}^{\eff}(\field)^{\op}): C\mapsto C^{\op},\quad C_i=(C^{\op})^i.
\]

\item
By negating the degrees, we have the following identification:
\[
K_b(\Chow_{\rat}^{\eff}(\field))\isomorphic K^b(\Chow_{\rat}^{\eff}(\field)): C\mapsto C^-,\quad C_i=(C^{-})^{-i}.
\]
\end{itemize}
\end{enumerate}

\subsubsection{Geometric motives}

We give a crash review on geometric motives \cite{Voe00,MVW06}:
\begin{enumerate}[wide,labelwidth=!,labelindent=0pt]
\item
Let $\FinCor(\field)$ be the category of finite correspondences over $\field$, and
$\PreSh(\FinCor(\field))$ be its category of abelian presheaves (\emph{presheaves with transfers}).
By Yoneda lemma, we obtain embeddings of $\Sm(\field)$ and $\FinCor(\field)$ into $\PreSh(\FinCor(\field))$:
\begin{equation}
\Sm(\field) \rightarrow \FinCor(\field)\rightarrow \PreSh(\FinCor(\field)): X\mapsto X\mapsto \mathbb{Z}_{\tr}(X).
\end{equation}

\item
Let $\Sh_{\Nis}(\FinCor(\field))$ be (abelian) category of \emph{Nisnevich sheaves with transfers} over $\field$.
Let $\DM_{\Nis}^{\eff,-}(\field)$ be the \emph{(tensor) triangulated category of effective (integral) motives over $\field$}, i.e. bounded above cochain complexes of Nisnevich sheaves with transfers, localized at 
\emph{$\mathbb{A}^1$-weak equivalences} generated by $\mathbb{Z}_{\tr}(X\times\mathbb{A}^1) \rightarrow \mathbb{Z}_{\tr}(X)$, $X\in\Sm(\field)$.
There are two functors
\begin{eqnarray}
&&\M:\Var(\field)\rightarrow \DM_{\Nis}^{\eff,-}(\field): X\mapsto \M(X)=\mathbb{Z}_{\tr}(X).~(\text{(effective) motives})\\
&&\M^c:\Var^{\prop}(\field)\rightarrow \DM_{\Nis}^{\eff,-}(\field): X\mapsto \M^c(X).~(\text{(effective) motives w/ compact supp.})
\end{eqnarray}
where $\Var^{\prop}(\field)$ denotes the full subcategory of $\Var(\field)$ with only proper morphisms.


\item
The \emph{(tensor) triangulated category $\DM_{\gm}^{\eff}(\field)$ of effective geometric motives over $\field$}
is the thick subcategory of $\DM_{\Nis}^{\eff,-}(\field)$ generated by $\M(X)$, $X\in\Sm(\field)$.
Recall that $\M^c(X)\in\DM_{\gm}^{\eff}(\field)$ for all $X\in\Var(\field)$ (see \cite[Cor.16.17]{MVW06}).

\item
Let $\mathbb{Z}(1):=\M^c(\mathbb{A}^1)[-2]=\coker(\M(\pt)\rightarrow\M(\mathbb{P}^1))[-2]\in\DM_{\gm}^{\eff}(\field)$. The \emph{Tate twist} functor is
\[
-\otimes\mathbb{Z}(1):\DM_{\Nis}^{\eff,-}(\field) \rightarrow \DM_{\Nis}^{\eff,-}(\field):\M\mapsto \M(1):=\M\otimes\mathbb{Z}(1).
\]
Denote $\M(m):=\M\otimes\mathbb{Z}(1)^{\otimes m}$. The \emph{(geometric) Lefschetz motive} is $\bfL:=\M^c(\mathbb{A}^1)=\mathbb{Z}(1)[2]$. 

\end{enumerate}

Among many other properties of effective geometric motives, we need the following:

\begin{theorem}[See e.g. \cite{MVW06}]\label{thm:geometric_motives}
The category $\DM_{\Nis}^{\eff,-}(\field)$ satisfies the following properties:
\begin{enumerate}[wide,labelwidth=!,labelindent=0pt]
\item
For all $X,Y\in\Var(\field)$, we have
\begin{eqnarray*}
&&\M(X\times\mathbb{A}^1)\isomorphic \M(X),\quad \M(X\times Y)\isomorphic \M(X)\otimes \M(Y).\\
&&\M^c(X\times\mathbb{A}^1)\isomorphic \M^c(X)(1)[2],\quad \M^c(X\times Y)\isomorphic \M^c(X)\otimes\M^c(Y).
\end{eqnarray*}

\item
(Triangle for $\M^c$) For any closed subvariety $i:Z\hookrightarrow X$ with open complement $j:U\hookrightarrow X$, there is an exact triangle:
\[
\M^c(Z)\xrightarrow[]{i_*} \M^c(X)\xrightarrow[]{j^*} \M^c(U)\rightarrow \M^c(Z)[1].
\]

\item
(Cancellation \cite{Voe10}) The Tate twist $-\otimes\mathbb{Z}(1):\DM_{\Nis}^{\eff,-}(\field) \rightarrow \DM_{\Nis}^{\eff,-}(\field)$ is \emph{fully faithful}.

 \item
(Chow motives) By \cite{Voe00}, there exists a covariant embedding
\begin{equation}
\iota:\Chow_{\rat}^{\eff}(\field)\rightarrow \DM_{\gm}^{\eff}(\field),
\end{equation}
such that $\iota(\M_{\rat}(X))=\M(X)=\M^c(X)$ for all $X\in\SmProj(\field)$.

\end{enumerate}
\end{theorem}

When $X$ is smooth quasi-projective, the motive with compact support $\M^c(X)$ admits a concrete description in terms of smooth projective varieties. 
Suppose $X$ is of pure dimension $d$. 
By Hironaka's resolution of singularities, a log compactification $(\overline{X},\partial X)$ always exists.
By blowing up further if necessary, we may further assume that $\partial X$ is \emph{very simple normal crossing}:  any nonempty finite intersection of the irreducible components of $\partial X$ is connected.
From now on, we fix a log compactification $(\overline{X},\partial X)$ with \emph{very simple normal crossing} divisor $\partial X$.
Say, $\partial X= \cup_{i=1}^n Y_i$ is the union of irreducible components.
For any subset $I\subset [n]=\{1,\text{\tiny$\cdots$},n\}$, denote
$Y_I:= \cap_{i\in I} Y_i$. Then, $Y_I\subset X$ is either empty or smooth connected.
For any integer $k\geq 1$, denote
\begin{equation}\label{eqn:strata_for_boundary_divisor}
Y^{(k)}:= \sqcup_{I\subset[n]:|I|=k} Y_I,
\end{equation}
and $Y^{(0)}=Y_{\emptyset}:=\overline{X}$. In particular, $\dim Y^{(k)}=\dim X - k =d-k$ (if nonempty). 

For any $1\leq j\leq k$, let $\delta_j:Y^{(k)}\rightarrow Y^{(k-1)}$ be the disjoint union of inclusions $Y_I\hookrightarrow Y_{I\setminus\{i_j\}}$, $I=\{i_1<\text{\tiny$\cdots$} <i_k\}\subset[n]$. Then
\begin{equation}\label{eqn:motivic_differential}
\partial=\partial^{(k)}:=\text{\tiny$\sum_{j=1}^k$}(-1)^{j-1}\delta_j: \M(Y^{(k)}) \rightarrow \M(Y^{(k-1)}).
\end{equation}
defines a morphism in $\DM_{\Nis}^{\eff,-}(\field)$.
It's direct to see that $\partial^2=0$.

In addition, observe that $\delta:\M(Y^{(1)})=\M^c(Y^{(1)})\rightarrow \M(\overline{X})=\M^c(\overline{X})$ factors through $\M^c(\partial X)$, 
so by Theorem \ref{thm:geometric_motives}.(2), the composition $\M(Y^{(1)})\xrightarrow[]{\partial} \M(\overline{X})=\M^c(\overline{X})\xrightarrow[]{\epsilon} \M^c(X)$ is zero.

\begin{lemma}\label{lem:motive_with_compact_support_of_a_smooth_quasi-projective_variety}
The total complex of
\[
\M(Y^{(d)})\xrightarrow[]{\partial} \text{\tiny$\cdots$}\xrightarrow[]{\partial} \M(Y^{(1)})\xrightarrow[]{\partial} \M(\overline{X}) \xrightarrow[]{\epsilon} \M^c(X)
\]
is acyclic in $\DM_{\Nis}^{\eff,-}(\field)$. In other words, $\M^c(X)$ is naturally isomorphic to the cochain complex
\[
\M(Y^{(\bullet)}):=[\M(Y^{(d)})\xrightarrow[]{\partial} \text{\tiny$\cdots$}\xrightarrow[]{\partial} \M(Y^{(1)})\xrightarrow[]{\partial} \M(Y^{(0)})],
\]
where $\M(Y^{(k)})$ lies in degree $-k$.
\end{lemma}

\begin{proof}
We prove by induction on $n$. The case $n=1$ is Theorem \ref{thm:geometric_motives}.(2). For the induction, suppose the statement holds for `$<n$', we prove the case `$n$'.

Denote $Y:=\partial X=\cup_{i=1}^n Y_i\subset \overline{X}$, $Y':=\cup_{i=1}^{n-1}\subset Y\subset\overline{X}$, and $X':=\overline{X}-Y'$. Define $Y'^{(k)}:=\sqcup_{I\subset[n-1],|I|=k} Y_I$.
Consider the following commutative diagram in $\DM_{\Nis}^{\eff,-}(\field)$
\[
\begin{tikzcd}[row sep=1.5pc, column sep=1.5pc]
\M(Y'^{(d)})\arrow[r,"{\partial}"] & \M(Y'^{(d-1)})\arrow[r,"{\partial}"] & \text{\tiny$\cdots$}\arrow[r,"{\partial}"] & \M(Y'^{(1)})\arrow[r,"{\partial}"] & \M(Y^{(0)})\arrow[r,"{\epsilon}"] & \M^c(X')\\
& \M(Y_n\cap Y'^{(d-1)})\arrow[u,"{\delta_d}"]\arrow[r,"{\partial}"] & \text{\tiny$\cdots$}\arrow[r,"{\partial}"] & \M(Y_n\cap Y'^{(1)})\arrow[r,"{\partial}"]\arrow[u,"{\delta_2}"] & \M(Y_n)\arrow[r,"{\epsilon}"]\arrow[u,"{\delta_1}"] & \M^c(Y_n\setminus Y')\arrow[u,"{i'_*}"]
\end{tikzcd}
\]
with each row a cochain complex. By inductive hypothesis, both of the two rows are acyclic in $\DM_{\Nis}^{\eff,-}(\field)$. In short, we may rewrite the commutative diagram as
\[
\begin{tikzcd}[row sep=1.5pc]
\M(Y'^{\bullet})\arrow[r,"{\epsilon}"] & \M^c(X')\\
\M(Y_n\cap Y'^{\bullet})\arrow[r,"{\epsilon}"]\arrow[u,"{\delta_{\bullet}}"] & \M^c(Y_n\setminus Y')\arrow[u,"{i'_*}"]
\end{tikzcd}
\]
In particular, the induced morphism $\epsilon:\cone(\delta_{\bullet})\rightarrow \cone(i'_*)$ is an isomorphism in $\DM_{\Nis}^{\eff,-}(\field)$.

Notice that $Y^{(k)}=Y'^{(k)}\sqcup (Y_n\cap Y'^{(k-1)})$, hence $\M(Y^{(k)})=\M(Y'^{(k)})\oplus\M(Y_n\cap Y'^{(k-1)})$. We then see that
$
\M(Y^{\bullet})=\cone(\delta_{\bullet}).
$
By Theorem \ref{thm:geometric_motives}.(2), we have a distinguished triangle in $\DM_{\Nis}^{\eff,-}(\field)$:
\[
\M^c(Y_n\setminus Y')\xrightarrow[]{i'_*} \M^c(X')\rightarrow \M^c(X) \rightarrow \M^c(Y_n\setminus Y')[1].
\]
This means the induced isomorphism is in fact
\[
\epsilon:\M(Y^{\bullet})=\cone(\delta_{\bullet})\xrightarrow[]{\isomorphic} \cone(i'_*)\isomorphic \M^c(X).
\]
We're done.
\end{proof}

\subsection{Motivic weight complexes}

We recall Gillet-Soul\'{e}'s motivic weight complexes. 

\begin{theorem}[{\cite[Thm.2]{GS96}}]\label{thm:weight_complex}
For any $\field$-variety $X$, there exists a \textbf{weight complex} $\weightcomplex(X)\in K_b(\Chow_{\rat}^{\eff}(\field))$, well-defined up to canonical isomorphism, such that:
\begin{enumerate}[wide,labelwidth=!,labelindent=0pt,label=(\roman*)]
\item
$\weightcomplex(X)$ is isomorphic to a bounded chain complex of the form
\[
\M_{\rat}(X_k)\rightarrow\text{\tiny$\cdots$}\rightarrow \M_{\rat}(X_1)\rightarrow \M_{\rat}(X_0),\quad X_i\in\SmProj(\field),
\]
where $\M_{\rat}(X_i)$ is in degree $i$, and $\dim(X_i)\leq \dim X -i$. In particular, $k\leq\dim X$.

\noindent{}In addition, if $X\in\SmProj(\field)$, then $\weightcomplex(X)=\M_{\rat}(X)$.

\item
$\weightcomplex(X)$ is covariant for proper morphisms, and contravariant for open inclusions.

\item
For any open subset $j:U\hookrightarrow X$ with closed complement $i:Z\hookrightarrow X$, there is a canonical triangle in $K_b(\Chow_{\rat}^{\eff}(\field))$:
\[
\weightcomplex(Z)\xrightarrow[]{i_*} \weightcomplex(X)\xrightarrow[]{j^*} \weightcomplex(U)\rightarrow \weightcomplex(Z)[1].
\]

\item
Any cover $X=A\cup B$ by two closed subvarieties gives a canonical triangle in $K_b(\Chow_{\rat}^{\eff}(\field))$:
\[
\weightcomplex(A\cap B)\rightarrow \weightcomplex(A)\oplus\weightcomplex(B)\rightarrow \weightcomplex(X)\rightarrow \weightcomplex(A\cap B)[1].
\]
Any cover $X=U\cup V$ by  two open subsets gives a canonical triangle in $K_b(\Chow_{\rat}^{\eff}(\field))$:
\[
\weightcomplex(X)\rightarrow \weightcomplex(U)\oplus\weightcomplex(V)\rightarrow \weightcomplex(U\cap V)\rightarrow \weightcomplex(X)[1].
\]

\item
For any $X,Y\in\Var(\field)$, have
\[
\weightcomplex(X\times Y)=\weightcomplex(X)\otimes\weightcomplex(Y).
\]
\end{enumerate}
\end{theorem}
In \cite{GS96}, the notation corresponds to the weight cochain complex $\weight(X)=\weightcomplex(X)^{\op}\in K^b(\Chow_{\rat}^{\eff}(\field)^{\op})$.
Another notation is $\weightcomplex^-(X)=(\weightcomplex(X))^-\in K^b(\Chow_{\rat}^{\eff}(\field))$.

The weight complex has an alternative description via Bondarko's weight complex functor.

\begin{theorem}\cite[3.3.1,6.3.1,6.4.2,6.6.2]{Bon09}\label{thm:weight_complex_functor}
There is a \emph{conservative} exact functor
\begin{equation}
t:\DM_{\gm}^{\eff}(\field)\rightarrow K^b(\Chow_{\rat}^{\eff}(\field)),
\end{equation}
called the \textbf{weight complex functor},
such that: 
\begin{enumerate}[wide,labelwidth=!,labelindent=0pt,label=(\alph*)]
\item
The composition $t\circ\iota:\Chow_{\rat}^{\eff}(\field)\hookrightarrow \DM_{\gm}^{\eff}(\field)\rightarrow K^b(\Chow_{\rat}^{\eff}(\field))$ is the obvious inclusion. 

\item
$t$ induces an isomorphism on the Grothendieck rings:
\[
K_0(t):K_0(\DM_{\gm}^{\eff}(\field))\xrightarrow[]{\isomorphic} K_0(\Chow_{\rat}^{\eff}(\field)).
\]

\item
There exists a natural isomorphism
\[
t\circ\M^c\isomorphic \weightcomplex^-:\Var(\field)\rightarrow K^b(\Chow_{\rat}^{\eff}(\field)).
\]
Here, by \textbf{Convention} \ref{convention:weight_complex}, $\weightcomplex^-(X)=(\weightcomplex(X))^-$.
\end{enumerate}
\end{theorem}

Observe that Lemma \ref{lem:motive_with_compact_support_of_a_smooth_quasi-projective_variety} and Theorem \ref{thm:weight_complex_functor} 
give an alternative proof of the following
\begin{proposition}[{\cite[Prop.3]{GS96}}]\label{prop:weight_complex_of_a_smooth_quasi-projective_variety}
If $X$ is a smooth quasi-projective $\field$-variety of pure dimension $d$ as in Lemma \ref{lem:motive_with_compact_support_of_a_smooth_quasi-projective_variety}, then $\weightcomplex(X)$ is isomorphic to the following chain complex in $K_b(\Chow_{\rat}^{\eff}(\field))$:
\[
\M_{\rat}(Y^{(d)})\xrightarrow[]{\partial} \text{\tiny$\cdots$}\xrightarrow[]{\partial} \M_{\rat}(Y^{(1)})\xrightarrow[]{\partial} \M_{\rat}(\overline{X}),
\]
where $M_{\rat}(Y^{(k)})$ is concentrated in homological degree $k$.
\end{proposition} 
\noindent{}Alternatively by \textbf{Convention} \ref{convention:weight_complex}, $\weight(X)=\weightcomplex(X)^{\op}$ is the cochain complex in $K^b(\Chow_{\rat}^{\eff}(\field)^{\op})$:
\begin{equation}
\M_{\rat}(\overline{X}) \xrightarrow[]{\partial^*} \M_{\rat}(Y^{(1)}) \xrightarrow[]{\partial^*}\text{\tiny$\cdots$} \xrightarrow[]{\partial^*} \M_{\rat}(Y^{(d)}),
\end{equation}
where $M_{\rat}(Y^{(k)})$ is concentrated in cohomological degree $k$.

\subsection{Weight cohomology with compact support}

Let $\realization:\Chow_{\rat}^{\eff}(\field)\rightarrow\cA$ be a contravariant additive functor into an abelian category, i.e. $\realization:\Chow_{\rat}^{\eff}(\field)^{\op}\rightarrow\cA$ is covariant.

\begin{definition}[Weight cohomology with compact support]
For any $X\in\Var(\field)$, the \emph{$a$-th weight cohomology with compact support} of $X$ associated to $\realization$ is
\[
H_{\weight,c}^a(X,\realization):=H^a(\realization(\weightcomplex(X)^{\op}))=H^a(\realization(\weight(X)))\in\cA.
\]
\end{definition}
By definition, the weight cohomology with compact support satisfies the following properties:
\begin{enumerate}[wide,labelwidth=!,labelindent=0pt]
\item
$H_{\weight,c}^a(X,\realization)$ is contravariant in $X$, and equals zero for $a<0$ or $a>\dim X$.

\item
For $X\in\SmProj(\field)$,
$H_{\weight,c}^a(X,\realization)=\realization(\M_{\rat}(X))$ for $a=0$, and vanishes for all $a\neq 0$.

\item
If $i:Z\hookrightarrow X$ is a closed subvariety with open complement $j:U\hookrightarrow X$, then there is a long exact sequence
\[
\cdots\rightarrow H_{\weight,c}^a(U,\realization)\xrightarrow[]{j_*} H_{\weight,c}^a(X,\realization)\xrightarrow[]{i^*} H_{\weight,c}^a(Z,\realization)\rightarrow H_{\weight,c}^{a+1}(U,\realization)\rightarrow\cdots.
\]
\end{enumerate}

The main example of interest is the singular cohomology with compact support:
Let $\field=\mathbb{C}$, and $b\in\mathbb{N}$.
For any $X\in\SmProj(\field)$ ($\field=\mathbb{C}$), let $H^*(X):=H^*(X(\mathbb{C}),\mathbb{Z})$ denote the singular cohomology with integer coefficients.

\begin{definition}[Singular weight cohomology with compact support]\label{def:Betti_weight_cohomology_with_compact_support}~
\begin{enumerate}[wide,labelwidth=!,labelindent=0pt,label=\roman*)]
\item
The \emph{$b$-th singular/Betti cohomology (with integer coefficients)} of effective pure Chow motives is 
the contravariant additive functor into abelian groups:
\begin{equation}
\realization_B^b:\Chow_{\rat}^{\eff}(\field)\rightarrow \cA=\Ab: \realization_B^b(X,p):=p^*H^b(X).
\end{equation}
In particular, for any $X\in\SmProj(\field)$, we have $\realization_B^b(X)=\realization_B^b(X,\id_X)=H^b(X)$.
By definition, $\realization_B^b(X,p)$ is always a finitely generated abelian group.

\item
For any $X\in\Var(\field)$, its \emph{$(a,b)$-th singular/Betti weight cohomology with compact support} is:
\[
H_{\weight,c}^{a,b}(X;\mathbb{Z}) := H_{\weight,c}^a(X,\realization_B^b)=H^a(\realization_B^b(W(X)))\in\Ab.
\]
For any commutative ring $A$,  $H_{\weight,c}^{a,b}(X;A)$ is defined similarly.
\end{enumerate}
\end{definition}

\begin{lemma}[{\cite[Thm.3]{GS96}}]\label{lem:weight_cohomology_vs_weight_filtration}
There is a canonical \emph{cohomological descent spectral sequence}
\[
E_2^{a,b}=H_{\weight,c}^{a,b}(X;A)~\Rightarrow~H_c^{a+b}(X(\mathbb{C});A).
\]
So, it defines a canonical increasing weight filtration on $H_c^k(X(\mathbb{C});A)$. When $A=\mathbb{Q}$, the spectral sequence degenerates at $E_2$ and we recover
Deligne's weight filtration:
\[
H_{\weight,c}^{a,b}(X;\mathbb{Z})\otimes\mathbb{Q} = \Gr_b^{\weight}H_c^{a+b}(X(\mathbb{C});\mathbb{Q}).
\]
\end{lemma}

\begin{remark}\label{rem:Betti_weight_cohomology_with_compact_support_for_a_smooth_quasi-projective_variety}
Let $X$ be a connected smooth quasi-projective variety of dimension $d$. 
Fix a log compactification $(\overline{X},\partial X)$ with very simple normal crossing boundary divisor.
Following the notation in (\ref{eqn:strata_for_boundary_divisor}), $\partial X=\cup_{i=1}^n Y_i$.
By Proposition \ref{prop:weight_complex_of_a_smooth_quasi-projective_variety}, $H_{\weight,c}^{a,b}(X;\mathbb{Z})$ is the $a$-th cohomology of the cochain complex:
\[
\realization_B^b(W(X))=[H^b(\overline{X})\xrightarrow[]{\partial^*} H^b(Y^{(1)})\xrightarrow[]{\partial^*}\text{\tiny$\cdots$}\xrightarrow[]{\partial^*} H^b(Y^{(d)})].
\]
\end{remark}

\section{Integral cohomology of dual boundary complexes is motivic}\label{subsec:integral_cohomology_of_dual_boundary_complexes_is_motivic}

We're ready to prove our main result.

\begin{proof}[Proof of Proposition \ref{prop:dual_boundary_complex_is_motivic}]
As in Remark \ref{rem:Betti_weight_cohomology_with_compact_support_for_a_smooth_quasi-projective_variety}, 
$H_{\weight,c}^{a,0}(X;\mathbb{Z})$ is the $a$-th cohomology of the cochain complex:
\[
\realization_B^0(W(X))=[H^0(\overline{X})\xrightarrow[]{\partial^*} H^0(Y^{(1)})\xrightarrow[]{\partial^*}\text{\tiny$\cdots$}\xrightarrow[]{\partial^*} H^0(Y^{(d)})].
\]
As $H^0(\overline{X})=\mathbb{Z}$, and $H^0(Y^{(i)})=\oplus_{I\subset[n]:|I|=i, Y_I\neq\emptyset}\mathbb{Z}$ for $i\geq 1$, we get an identification:
\[
\realization_B^0(W(X))\isomorphic \tilde{C}^{\bullet}(\mathbb{D}\partial X)[-1],
\]
with the latter the reduced simplicial cochain complex of $\mathbb{D}\partial X$. The result then follows.
\end{proof}

As an application, we obtain:
\begin{lemma}\label{lem:stably_invariance_of_motives_with_compact_support}
If $X,Y$ are stably isomorphic $\field$-varieties: $X\times\mathbb{A}^j\isomorphic Y\times\mathbb{A}^j$ for some $j\geq 0$, then 
\[
\M^c(X)\isomorphic \M^c(Y)\in\DM_{\Nis}^{\eff,-}(\field).
\]
In particular, $\weightcomplex(X)\isomorphic\weightcomplex(Y)\in K^b(\Chow_{\rat}^{\eff}(\field)^{\op})$, and hence $H_{\weight,c}^{a,b}(X;\mathbb{Z})\isomorphic H_{\weight,c}^{a,b}(Y;\mathbb{Z})$.
\end{lemma}

\begin{proof}
By Theorem \ref{thm:geometric_motives}.(1), the isomorphism $\phi:X\times\mathbb{A}^j\isomorphic Y\times\mathbb{A}^j$ induces an isomorphism
\[
\M^c(\phi):\M^c(X\times\mathbb{A}^j)=\M^c(X)(j)[2j] \xrightarrow[]{\isomorphic} \M^c(Y\times\mathbb{A}^j)=\M^c(Y)(j)[2j],
\]
hence equivalently,
$\M^c(\phi):\M^c(X)(j) \xrightarrow[]{\isomorphic} \M^c(Y)(j)$.
But by the Cancellation theorem \cite{Voe10} (see Theorem \ref{thm:geometric_motives}.(3)), we have a natural isomorphism
\[
-\otimes\mathbb{Z}(j): \Hom_{\DM_{\Nis}^{\eff,-}(\field)}(\M^c(X),\M^c(Y))\isomorphic \Hom_{\DM_{\Nis}^{\eff,-}(\field)}(\M^c(X)(j),\M^c(Y)(j))\ni \M^c(\phi).
\]
This shows that $\M^c(\phi)$ induces an isomorphism $\M^c(X)\isomorphic \M^c(Y)$, as desired.
\end{proof}

\begin{remark}
The equality $H_{\weight,c}^{a,b}(X)\isomorphic H_{\weight,c}^{a,b}(Y)$ in Lemma \ref{lem:stably_invariance_of_motives_with_compact_support}
admits an alternative proof without using geometric motives: 
It suffices to show
\begin{equation}
H_{\weight,c}^{a,b}(X\times\mathbb{A}^1;\mathbb{Z})\isomorphic H_{\weight,c}^{a,b-2}(X;\mathbb{Z}).
\end{equation}
By Proposition \ref{prop:weight_complex_of_a_smooth_quasi-projective_variety}, $\weight(\mathbb{A}^1)=[\M_{\rat}(\mathbb{P}^1)\xrightarrow[]{i^*}\M_{\rat}(\pt)]$, where $\M_{\rat}(\mathbb{P}^1)$ lies in degree $0$.
Observe that $H^*(\mathbb{P}^1)=\mathbb{Z}\oplus\mathbb{Z}[-2]\xrightarrow[]{i^*} H^*(\pt)=\mathbb{Z}$ with the obvious projection.
So, 
$H_{\weight,c}^{\bullet,\star}(\mathbb{A}^1)=\mathbb{Z}[0]\{-2\}$, 
by which we mean $\mathbb{Z}$ concentrated in bidegree $(a,b)=(0,2)$.
By Theorem \ref{thm:weight_complex} $(v)$, we have: 
\[
\weight(X\times\mathbb{A}^1)\isomorphic\weight(X)\otimes\weight(\mathbb{A}^1)\in K^b(\Chow_{\rat}^{\eff}(\field)^{\op}).
\]
It follows that we obtain the desired bi-graded isomorphism
\[
H_{\weight,c}^{\bullet,\star}(X\times\mathbb{A}^1;\mathbb{Z})\isomorphic H_{\weight,c}^{\bullet,\star}(X;\mathbb{Z})\otimes H_{\weight,c}^{\bullet,\star}(\mathbb{A}^1;\mathbb{Z})
=H_{\weight,c}^{\bullet,\star}(X;\mathbb{Z})[0]\{-2\}.
\]
\end{remark}

\section{Dual boundary complexes of varieties stably isomorphic to $\mathbb{G}_m^r\times \mathbb{A}^s$}

For the proof of Corollary \ref{cor:dual_boundary_complex_of_a_stably_torus_times_affine_space_is_contractible}, we will need some constraints on the fundamental group of a dual boundary complex.
This is pioneered by Koll\'{a}r-Xu \cite[\S 5]{KX16}. In our case, what we need is the following variation:
\begin{lemma}\label{lem:fundamental_group_for_bulk_vs_dual_boundary_complex}
Let $X$ be a smooth connected affine algebraic variety of dimension $d\geq 3$, with any log compactification $(\overline{X},D=\overline{X}-X)$. 
Then $\mathbb{D}\partial X$ is connected, and we have a natural surjection
\[
\pi_1(X)\twoheadrightarrow \pi_1(\overline{X})\simeq \pi_1(D)\twoheadrightarrow \pi_1(\mathbb{D}\partial X).
\]
\end{lemma}

\begin{proof}
This is a consequence of the Lefschetz hyperplane theorem. We give some details for completeness.
Say, $X$ is a proper closed subset of $\mathbb{A}^N$. Take a generic point $p\in\mathbb{A}^N\setminus X$, and define a squared-distance function $L_p:X\rightarrow\mathbb{R}$ as in \cite[p.41]{Mil63}. Then define $f:\overline{X}\rightarrow \mathbb{R}$ by: $f(x):=0$ if $x\in D$, and $f(x):=\frac{1}{L_p(x)}$. Now the same argument as in \cite[Thm.7.4]{Mil63} shows that $\pi_i(\overline{X},D)=0$ for all $i<d$.
In particular, $\pi_i(D)\simeq\pi_i(\overline{X}), i=0,1$ by the long exact sequence of homotopy groups. By definition, $\pi_0(\mathbb{D}\partial X)\simeq \pi_0(D)$. So, $\pi_0(\mathbb{D}\partial X)\simeq\pi_0(\overline{X})=0$. Now, by \cite[Lem.26]{KX16}, there is a natural surjection $\pi_1(D)\twoheadrightarrow \pi_1(\mathbb{D}\partial X)$. Finally, as $D$ has complex codimension $1$, hence real codimension $2$ in $\overline{X}$, we have a surjection $\pi_1(X=\overline{X}-D)\twoheadrightarrow \pi_1(\overline{X})$ by a perturbation argument. This completes the proof.
\end{proof}

Second, we need the following.
\begin{lemma}\label{lem:stable-cancellation_for_torus_times_affine_space_of_dimension_at_most_two}
If $Y$ is a complex variety stably isomorphic to $\mathbb{G}_m^r\times \mathbb{A}^s$ with $d+s\leq 2$, then $Y\isomorphic \mathbb{G}_m^r\times \mathbb{A}^s$.
\end{lemma}

\begin{proof}
By assumption, $s\leq 2$.
The cases $s = 0, 1$ follow from \cite[Prop.C.4, Cor.C.6]{Su26} (more generally, see \cite{AHE72}). 
If $s=2$, then $r = 0$, and the result follows from \cite{Fuj79,MS80}.
\end{proof}

Now, we come to our second goal.
\begin{proof}[Proof of Corollary \ref{cor:dual_boundary_complex_of_a_stably_torus_times_affine_space_is_contractible}]
Let $\sim$ denote homotopy equivalence.
Let $Y$ be a variety stably isomorphic to $\mathbb{G}_m^r\times \mathbb{A}^s$ with $s>0$.
Say, $Y\times \mathbb{A}^{\ell} \isomorphic \mathbb{G}_m^r\times \mathbb{A}^{s+\ell}$, and $d = r+s$.
Clearly, $Y$ is connected, affine, and smooth.

If $d\leq 2$, then $Y\isomorphic \mathbb{G}_m^r\times \mathbb{A}^s$ by Lemma \ref{lem:stable-cancellation_for_torus_times_affine_space_of_dimension_at_most_two}.
Hence, $\mathbb{D}\partial Y \sim \mathbb{D}\partial\mathbb{G}_m^r\star \mathbb{D}\partial\mathbb{A}^s = \pt$ (see \cite[E.g.3.4]{Su26}).

If $d\geq 3$, then by Lemma \ref{lem:fundamental_group_for_bulk_vs_dual_boundary_complex}, $\mathbb{D}\partial Y$ is connected, and we get a surjection
$\pi_1(Y)\twoheadrightarrow \pi_1(\mathbb{D}\partial Y)$.
Since $Y$ is homotopy equivalent to $\mathbb{G}_m^r$ as an analytic topological space, $\pi_1(Y)$ is abelian. It follows that so is $\pi_1(\mathbb{D}\partial Y)$. 
Hence by the Hurewicz theorem, we get a natural isomorphism $\pi_1(\mathbb{D}\partial Y)\isomorphic H_1(\mathbb{D}\partial Y,\mathbb{Z})$.
In addition, by Lemma \ref{lem:stably_invariance_of_motives_with_compact_support} and Proposition \ref{prop:dual_boundary_complex_is_motivic},
\[
\tilde{H}^i(\mathbb{D}\partial Y,\mathbb{Z})\isomorphic H_{\weight,c}^{i+1,0}(Y;\mathbb{Z}) \isomorphic H_{\weight,c}^{i+1,0}(\mathbb{G}_m^r\times \mathbb{A}^s;\mathbb{Z})\isomorphic \tilde{H}^i(\mathbb{D}\partial (\mathbb{G}_m^r\times \mathbb{A}^s),\mathbb{Z}) =0,
\]
for all $i\in\mathbb{N}$. Therefore, $H_i(\mathbb{D}\partial Y,\mathbb{Z})\isomorphic H_i(\pt,\mathbb{Z})$ and $\pi_1(\mathbb{D}\partial Y)=0$. Apply the Hurewicz theorem again, we then obtain that $\pi_i(\mathbb{D}\partial Y)\simeq 0$ for all $i\geq 0$. Finally, by the Whitehead theorem for CW complexes, we conclude that $\mathbb{D}\partial Y\sim\pt$, as desired.
\end{proof}



\end{document}